\documentclass{amsart} 
\usepackage{mathtools} 
\usepackage{amsfonts} 
\usepackage{amssymb}
\usepackage{xcolor}
\usepackage{amsmath} 
\usepackage{pb-diagram}
\usepackage[shortlabels]{enumitem}
\setlist{leftmargin=*}
\usepackage[colorlinks=false]{hyperref}
\usepackage[sorted]{amsrefs}

\long\def\comm#1\ent{}
\newtheorem{theorem}{Theorem}[section]

\newtheorem{claim}[theorem]{Claim}

\newtheorem{proposition}[theorem]{Proposition}
\newtheorem{lemma}[theorem]{Lemma}

\newtheorem{observation}[theorem]{Observation}
\newtheorem{question}[theorem]{Question}

\theoremstyle{definition}
\newtheorem{definition}[theorem]{Definition}

\newtheorem{remark}[theorem]{Remark}

\newcommand{\ie}{i.e.,\,}

\newcommand{\age}{{\rm Age}}
\newcommand{\Th}{{\rm Th}}

\newcommand{\SU}{{\rm SU}}

\begin{document}
\title{Asymptotic enumeration of $I_3$-free digraphs}
\author[A.~Aranda]{Andr\'es Aranda}
\address{Department of Mathematics and Statistics, University of Calgary, 2500 University Dr. NW, Calgary AB T2N 1N4, Canada}
\email{andres.arandalopez@ucalgary.ca}
\maketitle

\begin{abstract}
We prove that almost all digraphs not embedding an independent set of size 3 consist of two disjoint tournaments, and discuss connections with the theory of homogeneous simple structures.
\end{abstract}

Our main result can be stated informally as saying that almost all finite directed graphs in which any three vertices span at least one directed edge consist of two tournaments with some directed edges between them. This is a directed-graphs version of the following theorem by Erd\H os, Kleitman, and Rothschild (Theorem 2 in \cite{ErdosKlRo76}):

\begin{theorem}
	Let $T_n$ be the number of labelled triangle-free graphs on a set of $n$ vertices, and $S_n$ be the number of labelled bipartite graphs on $n$ vertices. Then \[T_n=S_n(1+o(\frac{1}{n})).\]
	\label{ErdosKleitmanRothschild}
\end{theorem}

So the proportion of triangle-free graphs on $n$ vertices that are not bipartite is negligible for large $n$.

Recall that a sentence $\sigma$ is \emph{almost surely true} (respectively, \emph{almost surely false}) if the fraction $\mu_n(\sigma)$ of structures with universe $\{0,\ldots,n-1\}$ satisfying $\sigma$ converges to 1 (0) as $n$ approaches infinity. Fagin \cite{fagin1976probabilities} proved:

\begin{theorem}
	Fix a relational language $L$. For every first-order sentence $\sigma$ over $L$, $\mu_n(\sigma)$ converges to 0 or to 1.
\end{theorem}

Given an $L$-sentence $\tau$ with $\mu_n(\tau)>0$ for all $n$, denote by  $\mu_n(\sigma|\tau)$ the conditional probability $\mu_n(\sigma|\tau)=\mu_n(\sigma\wedge\tau)/\mu_n(\tau)$. These conditional probabilities need not converge, but for some special cases they do converge. Given a relational language $L$ and appropriate $\tau$, let $T_{as}(L;\tau)$ be the set of $L$-sentences $\sigma$ with $\lim_{n\to\infty}\mu_n(\sigma|\tau)=1$. We call this the \emph{almost sure theory of $L$}. It follows from Gaifman's \cite{gaifman1964concerning} and Fagin's work that $T$ is consistent and complete when $\tau$ is $\forall x(x=x)$; Fagin proved in \cite{fagin1976probabilities} that $T$ is also consistent and complete in the cases where $L$ is the language $\{R\}$ and $\tau$ expresses one of the following:
\begin{enumerate}
\item{$R$ is a graph relation,}
\item{$R$ is a tournament predicate symbol. }
\end{enumerate}

We can think of Fra\"iss\'e's construction as a way to associate a complete first-order theory with infinite models (the theory of the Fra\"iss\'e limit) with a countable hereditary family of finite structures with the Joint Embedding Property and the Amalgamation Property; Fagin's theorem provides us with an alternative way of associating a first-order theory with a family of finite structures, namely the almost sure first-order theory of the language in question (possibly with some restrictions, represented by the sentences $\tau$).

In the studied cases of simple binary relational structures (the random graph, random $n$-graphs, the random tournament), the almost sure theory coincides with the theory of the Fra\"iss\'e limit. On the other hand, in the known cases where $\tau$ is such that the conditional probabilities $\mu_n(\sigma|\tau)$ converge, and the class of finite structures satisfying $\tau$ is the age of a non-simple homogeneous structure, the almost sure theory is simple (in fact, supersimple of \SU-rank 1). For example, it is known that the almost sure theory of triangle-free graphs is the theory of the Random Bipartite Graph (the proof has two stages, the first of which is Theorem \ref{ErdosKleitmanRothschild}; the second step is proving that almost all bipartite graphs satisfy the appropriate extension axioms); and whilst the generic triangle-free graph is not simple, the generic bipartite graph has supersimple theory of \SU-rank 1. Similarly, the almost-sure theory of partial orders is, by a result due to Kleitman and Rothschild \cite{kleitman1975asymptotic}, the theory of the generic 3-level partial order in which every element of the bottom level is less than every element of the top level; this theory is supersimple of \SU-rank 1. 

\begin{question}
Let $L$ be a binary relational language and $\tau$ an $L$-sentence such that the conditional probability $\mu_n(\sigma|\tau)$ converges for each $L$-sentence $\sigma$. Let $T_{\tau}$ be the almost sure theory of $L$ and $M\models T_{\tau}$, and suppose that $\age(M)$ is a Fra\"iss\'e class the limit of which has simple theory. Is it true that $T_\tau=\Th(M)$? Is it true that $T_\tau$ is always a simple theory?
\end{question}

\section{Remarks about $I_3$-free digraphs}
This section contains the definitions that we will use throughout the paper and a few observations about the universal homogeneous $I_3$-free digraph.

\begin{definition}~
\begin{enumerate}
	\item{A \emph{digraph} is a pair $(G,E)$ where $G$ is a set and $E$ is a subset of $G\times G$ such that for all $g\in G$ $(g,g)\notin E$ and $(g,g')\in E$ implies $(g',g)\notin E$. We will often denote a digraph $(G,E)$ by $G$ and write $g\rightarrow g'$ if $(g,g')\in E$.}
	\item{A digraph $G$ is $I_3$-free if every subset of three distinct vertices spans at least one arrow.}
	\item{A \emph{tournament} is a digraph $G$ in which for all distinct $x,y$, either $x\rightarrow y$ or $y\rightarrow x$ holds. A \emph{bitournament} is a digraph whose vertex set can be partitioned into two tournaments $T_1, T_2$ (we allow arrows from one tournament to the other).}
	\item{Given two vertices $x,y$ in a digraph $G$, we write $x\not\sim y$ if $x\not\rightarrow y$, $y\not\rightarrow x$, and $x\neq y$. If $v$ is a vertex in a digraph $G$, then $\Delta(v)=\{x\in G:x\not\sim v\}$. If $Q\subset G$, then $\Delta(Q)=(\bigcup_{v\in Q}\Delta(v))\setminus Q$.}
	\item{We denote the set of $I_3$-free digraphs on $\{0,\ldots,n-1\}$ by $F(n)$, and the set of bitournaments on the same set by $T(n)$.}
\end{enumerate}
\end{definition}

\begin{remark}
	Given an $I_3$-free digraph $(D,E)$, the graph $(D,Q(D))$, where $Q$ is the set of pairs $(d,d')\in D^2$ such that $(d,d'), (d',d)\notin E$ and $d'\neq d)$ is a triangle-free graph. Conversely, if we start with a triangle-free graph $G$, any orientation of the complement of $G$ is an $I_3$-free digraph.
\label{RmkGraphs}
\end{remark}

\begin{proposition}
	There exists a universal homogeneous $I_3$-free digraph $\mathcal D$ and it is a primitive structure.
\end{proposition}
\begin{proof}
	We will show that the family $\mathcal C$ of all finite $I_3$-free digraph satisfies Fra\"iss\'e's conditions. It is clear that $\mathcal C$ is countable (up to isomorphism) and closed under induced substructures. Given two structures $A,B\in\mathcal C$, we can embed both $A$ and $B$ in the structure defined on $A\times\{0\}\cup B\times\{1\}$ where for all $b\in B$ and all $a\in A$ we have $R((a,0),(b,1))$. The amalgamation property follows from the fact that given an amalgamation problem $f_1:A\rightarrow B$ and $g_1:A\rightarrow C$, let $D$ be $(B\times\{0\}\cup C\times\{1\})/\sim$, where $(b,0)\sim(c,1)$ if there exists $a\in A$ such that $f_1(a)=b$ and $g_1(a)=c$, and define a digraph relation on $D$ by $((p,i)/\sim)\rightarrow((q,j)/\sim)$ if there exist representatives of the classes that are related in $B$ or $C$, or if that condition fails and $i<j$. 

If $\mathcal D$ were imprimitive, then the reflexive closures of $\not\sim$ or the relation $x\rightarrow y\vee y\rightarrow x$ would define an equivalence relation on $\mathcal D$, by quantifier elimination. But these relations are clearly not transitive.
\end{proof}

\begin{proposition}
	The theory of the universal homogeneous $I_3$-free digraph is not simple.
\end{proposition}
\begin{proof}
	We will prove that the formula $\psi(x,a,b)=x\not\sim a\wedge x\not\sim b$ has the TP2. Let $\{(a^i_j,b^i_j):i,j\in\omega\}$ be an array of parameters such that $c^i_s\rightarrow c^i_t$ for $s<t$ and $c\in\{a,b\}$, $a^i_s\rightarrow b^i_t$ if $s\leq t$ and $a^i_s\not\sim a^j_t$ for $t<s$, and there are no other pairs satisfying $c^i_s\not\sim d^j_t$ ($c,d\in\{a,b\}$).
Any such array of parameters can be embedded into the universal homogeneous $I_3$-free digraph as elements from different levels $L_i=\{(a^i_j,b^i_j):j\in\omega\}$ are in a directed edge, and therefore no $I_3$ embeds into any level. Each level $L_i$ witnesses 2-dividing for $\psi$, and each branch is a tournament. Therefore, $\psi$ has the TP2.
\end{proof}

\begin{remark}
{\rm It is tempting to argue that given an $I_3$-free digraph, the associated graph obtained as in Remark \ref{RmkGraphs} is almost always a bipartite graph, and so an orientation of its complement will be a bitournament. But formalising this argument is not as straightforward as it seems.}
\end{remark}

\section{Asymptotic enumeration of $I_3$-free digraphs}
In this section we prove our main theorem:
\begin{theorem}\label{ThmMain}
$|F(n)|=|T(n)|(1+o(1))$
\end{theorem}

The general strategy we will follow consists of breaking up the set $F(n)$ into four parts: the bitournaments and three classes $A(n),B(n),C(n)$. We prove that as $n$ tends to infinity the proportion of $I_3$-free digraphs in $A(n)\cup B(n)\cup C(n)$ becomes negligible. All our logarithms are base 2, and when making assertions of the type $n=\log m$, where $n$ is an integer, by $\log m$ we mean the integral part of $\log m$.

\begin{definition}~
	\begin{enumerate}
		\item{$A(n)=\{\Gamma\in F(n):\exists v\in\Gamma(|\Delta(v)|\leq\log(n))\}$}
		\item{$B(n)=\{\Gamma\in F(n)\setminus A(n):\exists v\in\Gamma\exists Q\subset\Delta(v)(|Q|=\log(n)\wedge|\Delta(Q)|\leq(1/2-1/10^6)n)\}$}
		\item{$C(n)=\{\Gamma\in F(n)\setminus(A(n)\cup B(n)):\exists x,y\in \Gamma(x\not\sim y\wedge\exists Q_x\subseteq\Delta(x),Q_y\subseteq\Delta(y)(|Q_x|=|Q_y|=\log(n)\wedge|\Delta(Q_x)\cap\Delta(Q_y)|\geq n/100))\}$}
	\end{enumerate}
\end{definition}

Our proof of Theorem \ref{ThmMain} is an amalgamation of the arguments in \cite{PST2001} and \cite{PSS2002}, adapted for directed graphs.

\begin{observation}
	Let $G$ be an $I_3$-free digraph and $v\in G$. Then $\Delta(v)$ is a tournament, $v\in\Delta(\Delta(v))$, and $\Delta(v)\cap \Delta(\Delta(v))=\varnothing$.
\label{Obvious}
\end{observation}
\begin{proof}
	There can be no undirected arcs between any elements of $\Delta(v)$ as any such pair would form an $I_3$ with $v$. Take any $y\in\Delta(v)$. Then $y\not\sim v$, so $v\in\Delta(y)\subset\Delta(\Delta(v))$. And if $x\in\Delta(v)\cap\Delta(\Delta(v))$, then $x\not\sim v$ and $x\not\sim y$ for all $y\in\Delta(v)$, so $xyv$ forms an $I_3$.
\end{proof}
\begin{definition}
	A \emph{pinwheel} on $n$ vertices $v_0,\ldots,v_{n-1}$ is a digraph in which $v_i\not\sim v_{i+1}$ (addition is modulo $n$) for each $i\in n$. Equivalently, it is an orientation of the complement of a Hamiltonian graph on $n$ vertices. We will abuse notation and denote a pinwheel by $C_n$ even though there are several isomorphism types of pinwheels of the same size.
\end{definition}
\begin{lemma}
	If $n$ is sufficiently large, then $F(n)\subseteq T(n)\cup A(n)\cup B(n)\cup C(n)$\label{Fncontained}
\end{lemma}
\begin{proof}
	Suppose for a contradiction that for all $n$, $F(n)$ is a proper superset of $T(n)\cup A(n)\cup B(n)\cup C(n)$, and let $\Gamma\in F(n)\setminus(T(n)\cup A(n)\cup B(n)\cup C(n))$. This means that every vertex $v$ in $\Gamma$, $|\Delta(v)|>\log(n)$ and all nonempty subsets $Q$ of $\Delta(v)$ of size $\log n$ satisfy $|\Delta(Q)|>(1/2-1/10^6)n$.  As $\Gamma\not\in C(n)$, if $x\not\sim y$ and $x\neq y$, then $|\Delta(Q_x)\cap\Delta(Q_y)|<n/100$, where $Q_x$ and $Q_y$ are any subsets of $\Delta(x),\Delta(y)$ of size $\log(n)$.
\begin{claim}
	$\Gamma$ contains no pinwheels $C_5,C_7$ or $C_9$. \label{nopinwheels}
\end{claim}
\begin{proof}
The idea of the proof is the same in all cases: if we had a pinwheel on $\{v_0,\ldots,v_{2m}\}$ for $m=2,3,4$, then as $\Gamma$ is not in $B(n)\cap A(n)$ we know that there is a subset $Q_{v_i}$ of $\Delta(v_i)$ of size $\log n$ such that $R_{v_i}\mathrel{\mathop:}=\Delta(Q_{v_i})$ contains approximately half the vertices of the digraph. This implies that the $R_{v_i}$ have large intersection for $i$ even (odd), so the only way to satisfy that condition is if $R_{v_0}$ contains almost all the vertices of the digraph, but then there are not enough vertices left for $R_{v_1}$. We present the formal proofs next.

	Suppose that there is a $C_5$ on a set of vertices $\{v_0,\ldots,v_4\}$. Denote by $R_{v_i}$ the set $\Delta(Q_{v_i})$, where $Q_{v_i}\subseteq\Delta(v_i)$ is of size $\log(n)$. For any distinct $x,y$ with $x\not\sim y$, \[|R_x\cup R_y|=|R_x|+|R_y|-|R_x\cap R_y|\geq n(1-2/10^6-1/100)\] and \[|\bar R_x\cap\bar R_y|=n-|R_x\cup R_y|\leq (2/10^6+1/100)n,\]
where $\bar R_x$ stands for the complement of $R_x$ in the vertex set of $\Gamma$. Notice that as $|R_{v_1}\cap R_{v_2}|<n/100$ and $|\bar R_{v_0}\cap\bar R_{v_1}|\leq n(2/10^6+1/100)$, then 
\[
	|R_{v_0}\cap R_{v_2}|\geq |R_{v_2}|-|R_{v_2}\cap R_{v_1}|-|\bar R_{v_0}\cap\bar R_{v_1}|\geq n(\frac{1}{2}-\frac{3}{10^6}-\frac{2}{100}).
\]
Similarly, $|R_{v_0}\cap R_{v_3}|\geq n(\frac{1}{2}-\frac{3}{10^6}-\frac{2}{100})$. This gives us
\[
	\begin{split}
		|R_{v_0}|&=|R_{v_0}\cap R_{v_2}|+|R_{v_0}\cap R_{v_3}|-|R_{v_0}\cap R_{v_2}\cap R_{v_3}|+|R_{v_0}\cap\bar R_{v_2}\cap\bar R_{v_3}|\geq\\&\geq n(\frac{1}{2}-\frac{3}{10^6}-\frac{2}{100})+n(\frac{1}{2}-\frac{3}{10^6}-\frac{2}{100})-\frac{n}{100}=\\&=n(1-\frac{6}{10^6}-\frac{5}{100})
	\end{split}
\]
So $R(v_0)$ is almost all the digraph. Now, 
\[
	\begin{split}
	|R_{v_1}|&=|R_{v_1}\cap\bar R_{v_0}|+|R_{v_1}\cap R_{v_0}|\leq\\&\leq n(\frac{6}{10^6}+\frac{5}{100})+\frac{n}{100}=\\&=n(\frac{6}{10^6}+\frac{6}{100})<n(\frac{1}{2}-\frac{1}{10^6}),
\end{split}
\]
which contradicts $\Gamma\not\in B(n)$.

Suppose now that we have a pinwheel on 7 vertices $v_0,\ldots,v_6$. Our estimate for $|R_{v_0}\cap R_{v_2}|$ is still valid, and by the same argument we know $|R_{v_0}\cap R_{v_5}|>n(1/2-3/10^6-2/100)$. Now we estimate $|R_{v_0}\setminus R_{v_3}|$ (the calculations hold for $|R_{v_0}\setminus R_{v_4}|$ as well).

\[
	\begin{split}
		|R_{v_0}\cap R_{v_3}|&=|R_{v_0}\cap R_{v_2}\cap R_{v_3}|+|R_{v_0}\cap R_{v_3}\cap\bar R_{v_2}|\leq\\&\leq |R_{v_3}\cap R_{v_2}|+|R_{v_0}\setminus R_{v_2}|<\\&<\frac{n}{100}+|R_{v_0}|-|R_{v_0}\cap R_{v_2}|<\\&<\frac{n}{100}-n(\frac{1}{2}-\frac{3}{10^6}-\frac{2}{100})+|R_{v_0}|
	\end{split}
\]

Therefore, $|R_{v_0}\setminus R_{v_3}|>n(1/2-3/10^6-3/100)$. Similarly, $|R_{v_0}\setminus R_{v_4}|>n(1/2-3/10^6-3/100)$.
Now we use this information to get a new estimate of $|R_{v_0}|$.
\[
	\begin{split}
		|R_{v_0}|&=|R_{v_0}\setminus R_{v_3}|+|R_{v_0}\setminus R_{v_4}|-|R_{v_0}\cap\bar R_{v_3}\cap\bar R_{v_4}|+|R_{v_0}\cap R_{v_3}\cap R_{v_4}|\\
&>2n(\frac{1}{2}-\frac{3}{10^6}-\frac{3}{100})-n(\frac{2}{10^6}+\frac{1}{100})\geq\\
&\geq n(1-\frac{8}{10^6}-\frac{7}{100})
	\end{split}
\]
Again, $R_{v_0}$ contains almost all the vertices in $\Gamma$. As before, this contradicts $\Gamma\not\in B(n)$:
\[
	|R_{v_1}|=|R_{v_1}\cap\bar R_{v_0}|+|R_{v_1}\cap R_{v_0}|<n(\frac{8}{10^6}+\frac{8}{100})<n(\frac{1}{2}-\frac{1}{10^6})
\]
Finally, suppose that there is a pinwheel on nine vertices in $\Gamma$. We know that $|R_{v_0}\cap\bar R_{v_3}|\geq n(1/2-3/10^6-3/100)$ and $|R_{v_0}\cap\bar R_{v_3}\cap\bar R_{v_4}|\leq n(2/10^6+1/100)$. From this, we derive
\[
	\begin{split}
		|R_{v_0}\cap\bar R_{v_3}\cap R_{v_4}|&=|R_{v_0}\cap\bar R_{v_3}|-|R_{v_0}\cap\bar R_{v_3}\cap\bar R_{v_4}|>\\&>n(\frac{1}{2}-\frac{3}{10^6}-\frac{3}{100})-n(\frac{2}{10^6}+\frac{1}{100})=\\&=n(\frac{1}{2}-\frac{5}{10^6}-\frac{4}{100})
	\end{split}
\]

It follows that $|R_{v_0}\cap R_{v_4}|>n(\frac{1}{2}-\frac{5}{10^6}-\frac{4}{100})$, and by the same argument (going down the other side of the pinwheel), $|R_{v_0}\cap R_{v_5}|>n(\frac{1}{2}-\frac{5}{10^6}-\frac{4}{100})$. It follows that
\[
	\begin{split}
		|R_{v_0}|&=|R_{v_0}\cap R_{v_4}|+|R_{v_0}\cap R_{v_5}|-|R_{v_0}\cap R_{v_4}\cap R_{v_5}|+|R_{v_0}\cap\bar R_{v_4}\cap\bar R_{v_5}|\geq\\&\geq2n(\frac{1}{2}-\frac{5}{10^6}-\frac{4}{100})-\frac{n}{100}=n(1-\frac{1}{10^5}-\frac{9}{100})
	\end{split}
\]
As a consequence, $|\bar R_{v_0}|<n(\frac{1}{10^5}+\frac{9}{100})$. Therefore,
\[
	\begin{split}
		|R_{v_1}|&=|R_{v_1}\cap\bar R_{v_0}|+|R_{v_1}\cap R_{v_0}|<n(\frac{1}{10^5}+\frac{9}{100})+\frac{n}{100}=\\&=n(\frac{1}{10^5}+\frac{1}{10})<n(\frac{1}{2}-\frac{1}{10^6}),
	\end{split}
\]
contradicting $\Gamma\not\in B(n)$.
\end{proof}
	Now we describe how to find a partition of $\Gamma$ into two tournaments. For readability, we will use $U_v$ to denote $\Delta(\Delta(v))$. Take an arbitrary non-arc $x\not\sim y$; then as $\Gamma$ is $I_3$-free, $\Delta(x)\cap\Delta(y)=\varnothing$ and $U_x\cap U_y=\varnothing$ because  any $z\in U_x\cap U_y$ would form a $C_5$ with $x,y,x',y'$, for some $x'\in\Delta(x)$ and $y'\in\Delta(y)$. For the same reason (no $C_5$), $U_x$ and $U_y$ are tournaments. Let $W=V(\Gamma)\setminus(\Delta(x)\cup U_x\cup\Delta(y)\cup U_y)$, and $W_x=\{v\in W:R_v\cap R_x\neq\varnothing\}$, $W_y=\{v\in W:R_v\cap R_y\neq\varnothing\}$; again $W_x\cap W_y=\varnothing$ because there are no $C_9$s in $\Gamma$. 

We know that $|R_x\cup R_y|\geq n(1-2/10^6-1/100)$, and since $\Gamma\not\in B(n)$, for all $v\in W$ we have $|R_v|\geq n(1/2-1/10^6)$; therefore, $R_v\cap(R_x\cup R_y)\neq\varnothing$ and every vertex in $W$ is in $W_x$ or $W_y$. Our partition consists of $W_x\cup U_x\cup\Delta(y)$ and $W_y\cup U_y\cup\Delta(x)$. 

We claim that $W_x\cup U_x\cup\Delta(y)$ is a tournament. By Observation \ref{Obvious}, $\Delta(x)$ is a tournament. 

Now consider $w\in\Delta(x)$ and $w'\in U_y$. We argued before that $U_x\cap U_y=\varnothing$, so $U_y\subseteq V(\Gamma)\setminus U_x$, and therefore $U_y\subseteq\bigcap_{v\in\Delta(x)}(V(\Gamma)\setminus\Delta(v))$, so there is a directed edge between $w$ and $w'$. Thus, $U_y\cup\Delta(x)$ is a tournament.

If $w$ and $w'$ are in $W_x$, then there exist $p\in\Delta(v), p'\in\Delta(v'), q,q'\in R_x,r,r'\in\Delta(x)$. From all these vertices, $p\neq p'$ because the digraph is $I_3$-free. So we have $w\not\sim p\not\sim q$ and $w'\not\sim p'\not\sim q'$. If $q=q'$, then a directed edge is forced between $w$ and $w'$ because $\Gamma$ is $C_5$-free. Similarly, an edge is forced if $r\neq r'$ because $\Gamma$ is $C_7$-free. Finally, even if all the vertices are distinct, an edge is forced because $r,r'\in\Delta(x)$, so a $C_9$ would be formed if $w\not\sim w'$.

Finally, suppose for a contradiction that $w\in W_x, w'\in U_x$ and $w\not\sim u$. Then there exist $q_w\in Q_w$ and $r_w\in R_w\cap R_x$ such that $w\not\sim q_w\not\sim r_w$. We also have either a $\not\sim$-path of length 2 $r_w\not\sim v\not\sim u$ with $v\in\Delta(x)$ or a $\not\sim$-path of length 4 $r_w\not\sim d\not\sim x\not\sim d'\not\sim u$ with $d,d'\in\Delta(x)$; in the first case we get a $C_5$ and in the second, a $C_7$, contradicting in any case Claim \ref{nopinwheels}. Therefore, $\Gamma$ is a bitournament, contradiction.
\end{proof}
\begin{lemma}
	$|T(n+1)|\geq6^{n/2}|T(n)|$
\end{lemma}
\begin{proof}
	There are $|T(n)|$ bitournaments on $[n]$. From a bitournament $T$ on $[n]$, we can build a bitournament on $[n+1]$ by adding the vertex $n+1$ to the smaller of the tournaments in a given partition of $T$ into two tournaments, which is of size at most $n/2$. Now we connect the vertex to the rest of the digraph: we need to make at least $3^{n/2}$ choices to connect it to the other tournament and at most $2^{n/2}$ choices to connect it to the smaller tournament. In total, at least $6^{n/2}$ choices for each tournament in $T(n)$, and the result follows.
\end{proof}

We wish to prove that the sets $A(n)$, $B(n)$, and $C(n)$ are negligible in size when compared to $T(n)$. The next step is to find bounds for their sizes relative to that of $F(n)$.
\begin{lemma}
	For sufficiently large $n$, $\log(\frac{|A(n)|}{|F(n-1)|})\leq n+\log^2n+ \log n-1$\label{An}
\end{lemma}
\begin{proof}
	To construct a digraph in $A(n)$, we need to 
		\begin{enumerate}
			\item{Select a vertex $v$ that will satisfy the condition in the definition of $A(n)$: $n$ possible choices;}
			\item{Select the neighbourhood $\Delta(v)$ of size at most $\log n$: $\sum_{i=0}^{\log n}{{n-1}\choose{i}}$ choices;}
			\item{Choose a digraph structure on $[n]\setminus\{v\}$: $|F(n-1)|$ choices;}
			\item{Connect $v$ to $[n]\setminus\Delta(v)$: at most $2^{n-1}$ choices;}
		\end{enumerate}
	In total, this gives the following estimates:
\begin{equation*}
	\begin{split}
		|A(n)|&\leq n(\sum_{i=0}^{\log n}{{n-1}\choose{i}})2^{n-1}|F(n-1)|\\
			&\leq nn^{\log n}2^{n-1}|F(n-1)|\\
	\end{split}
\end{equation*}
	So 
	\[
	\begin{split}
		\log(\frac{|A(n)|}{|F(n-1)|})&\leq\log n+\log^2n+n-1,
	\end{split}
	\]
as desired.
\end{proof}
\begin{lemma}
	For sufficiently large $n$, $\log(\frac{|B(n)|}{|F(n-\log n)|})\leq\beta n\log n+n+\frac{3}{2}\log^2n-\frac{1}{2}\log n$, where $\beta=\frac{1+\alpha}{2}+\frac{1-\alpha}{10^6}$, and $\alpha=\log3$.\label{Bn}
\end{lemma}
\begin{proof}
	All the digraphs in $B(n)$ can be constructed as follows:
	\begin{enumerate}
		\item{Choose a set $Q$ of size $\log n$: ${n}\choose{\log n}$ choices;}
		\item{Choose a tournament structure on $Q$: $2^{{{\log n}\choose{2}}}$ choices;}
		\item{Choose a digraph structure on $[n]\setminus Q$: $|F(n-\log n)|$ choices;}
		\item{Choose $R=\Delta(Q)$: at most $2^n$ choices;}
		\item{Connect $Q$ to $R$: $3^{(\log n)|R|}$ choices;}
		\item{Connect $Q$ to $[n]\setminus R$: $2^{(\log n)|[n]\setminus R|}$ choices}
	\end{enumerate}
So we have
\begin{equation}
		|B(n)|\leq{{n}\choose{\log n}}2^{{{\log n}\choose{2}}}|F(n-\log n)|2^n3^{\log n|R|}2^{\log n|[n]\setminus R|}
\label{Asterisk2}
\end{equation}
From this expression, the factor $3^{\log n|R|}2^{\log n|[n]\setminus R|}$ depends on the size of $R$. We claim that $3^{\log n|R|}2^{\log n|[n]\setminus R|}$, and therefore the expression \ref{Asterisk2}, is maximised when $|R|$ is maximal, \ie $|R|=n(1/2-1/10^6)$.
\[
\begin{split}
\log(3^{\log n|R|}2^{\log n|[n]\setminus R|})&=\alpha\log n|R|+\log n(n-|R|)=\\
&=n\log n+(\alpha-1)|R|\log n
\end{split}
\]
This expression is, as a function of $|R|$, a linear polynomial with positive slope $(\alpha-1)$. Therefore, the value of the expression in equation \ref{Asterisk2} is maximal when $|R|=n(1/2-1/10^6)$ is maximal, as claimed. Let us continue with the calculations:
\[
\begin{split}
|B(n)|&\leq{{n}\choose{\log n}}2^{{{\log n}\choose{2}}}|F(n-\log n)|2^n3^{\log n|R|}2^{\log n|[n]\setminus R|}\leq\\
	&\leq{{n}\choose{\log n}}2^{{{\log n}\choose{2}}}|F(n-\log n)|2^n3^{n\log n(\frac{1}{2}-\frac{1}{10^6})}2^{n\log n(\frac{1}{2}+\frac{1}{10^6})}=\\
	&={{n}\choose{\log n}}2^{{{\log n}\choose{2}}+n+n\log n(\frac{1}{2}+\frac{1}{10^6})}|F(n-\log n)|3^{n\log n(\frac{1}{2}-\frac{1}{10^6})}
\end{split}
\]
Therefore,
\[
	\begin{split}
		\log(\frac{|B(n)|}{|F(n-\log n)|})&\leq\log{{n}\choose{\log n}}+{{\log n}\choose{2}}+n+n\log n(\frac{1}{2}+\frac{1}{10^6})+\\
&\hspace{1cm}+\alpha n\log n(\frac{1}{2}-\frac{1}{10^6})\leq\\
		&\leq\log^2n+\frac{\log^2n-\log n}{2}+n+n\log n(\frac{1}{2}(\alpha+1)+\frac{1}{10^6}(1-\alpha))=\\
	&=\beta n\log n+n+\frac{3}{2}\log^2n-\frac{1}{2}\log n.
	\end{split}
\]
\end{proof}
\begin{lemma}
	For large enough $n$, $\log(\frac{|C(n)|}{|F(n-2)|})\leq\gamma n+2\log^2n+2\log n$, where $\gamma=1+\frac{4}{10^6}+\frac{3}{100}+\alpha(1-\frac{2}{10^6}-\frac{2}{100})$\label{Cn}
\end{lemma}
\begin{proof}
	Counting the elements in $C(n)$ is harder than counting $B(n)$ or $A(n)$, so we will give a rougher bound. All the elements in $C(n)$ can be found in the following way:
	\begin{enumerate}
		\item{Choose two elements $x,y$, which will satisfy $x\not\sim y$: $n\times (n-1)<n^2$ choices.}
		\item{Choose an $I_3$-free structure for $[n]\setminus{x,y}$: $|F(n-2)|$ options}
		\item{Choose neighbourhoods $Q_x$, $Q_y$ in $\Delta(x),\Delta(y)$ of size $\log n$. The $\not\sim$-neighbourhoods of $x$ and $y$ are disjoint because the digraph is $I_3$-free. So we have ${{n-1}\choose{\log n}}{{n-2-\log n}\choose{\log n}}\leq{{n-2}\choose{\log n}}^2\leq n^{2\log n}$ choices. Notice that at this point the neighbourhoods $R_x,R_y$ of $Q_x$ and $Q_y$ are determined by the $I_3$-free structure for $[n]\setminus\{x,y\}$, but we will only count those cases in which $|R_x\cap R_y|\geq\frac{n}{100}$.}
		\item{Connect $x,y$ to $[n]\setminus\{x,y\}$: We have already decided how to connect $x,y$ to $Q_x,Q_y$, so we need to decide:
			\begin{enumerate}
				\item{If $u\in R_x\cap R_y$, then there are only 4 possible ways to connect $x,y$ to $u$.}
				\item{If $u\in R_x\setminus R_y$ or $u\in R_y\setminus R_x$, then there are 6 possible ways to connect $x,y$ to $u$.}
				\item{If $u\in$ the complement of $R_x\cup R_y$, there are 8 ways to connect $x,y$ to $u$.}
			\end{enumerate}
		Therefore, we have 
\begin{equation}
4^{|R_x\cap R_y|}6^{|R_x\setminus R_y|+|R_y\setminus R_x|}8^{n-|R_x\cup R_y|}
\label{Asterisk}
\end{equation}
		choices to make at this point. We claim that the expression \ref{Asterisk} is maximised when $|R_x\cap R_y|$ and $|R_x\cup R_y|$ are minimised. 
\[
\begin{split}
&\log(4^{|R_x\cap R_y|}6^{|R_x|+|R_y|-2|R_x\cap R_y|}8^{n-|R_x|-|R_y|+|R_x\cap R_y|})=\\
&2|R_x\cap R_y|+(1+\alpha)(|R_x|+|R_y|-2|R_x\cap R_y|)+3(n-|R_x|-|R_y|+|R_x\cap R_y|)\\
&=3n+(3-2\alpha)|R_x\cap R_y|+(\alpha-2)(|R_x|+|R_y|).
\end{split}
\]

This is a linear polynomial in variables $|R_x\cap R_y|$, $|R_x|$,$|R_y|$, and is clearly maximal (in $(0,n]^3$) when the variables are minimised, as $3-2\alpha$ and $\alpha-2$ are both negative. By hypothesis, this happens when $|R_x\cap R_y|=\frac{1}{100} n$ and $|R_x|=|R_y|=n(1/2-1/10^6)$. Therefore, we have at most

\[
\begin{split}
	&4^{|R_x\cap R_y|}6^{|R_x\setminus R_y|+|R_y\setminus R_x|}8^{n-|R_x\cup R_y|}=\\
	&=4^{|R_x\cap R_y|}6^{|R_x|+|R_y|-2|R_x\cap R_y|}8^{n-(|R_x|+|R_y|-|R_x\cap R_y|)}\leq\\
	&\leq4^{\frac{1}{100}n}6^{n(1-\frac{2}{10^6}-\frac{2}{100})}8^{n(\frac{2}{10^6}+\frac{1}{100})}=\\
	&=2^{\frac{2}{100}n}2^{n\log6(1-\frac{2}{10^6}-\frac{2}{100})}2^{3n(\frac{2}{10^6}+\frac{1}{100})}=\\
	&=2^{n\log6(1-\frac{2}{10^6}-\frac{2}{100})+\frac{2}{100}n+3n(\frac{2}{10^6}+\frac{1}{100})}=\\
	&=2^{n(1+\alpha)(1-\frac{2}{10^6}-\frac{2}{100})+\frac{2}{100}n+3n(\frac{2}{10^6}+\frac{1}{100})}=\\
	&=2^{n(1+\frac{4}{10^6}+\frac{3}{100}+\alpha(1-\frac{2}{10^6}-\frac{2}{100}))}=\\
	&=2^{\gamma n}
\end{split}
\]
}

ways to connect $x,y$ to the rest of the digraph.
\end{enumerate}
In total, this gives us
\[\frac{|C(n)|}{|F(n-2)|}\leq n^2n^{2\log n}2^{\gamma n}\]
So
\[\log(\frac{|C(n)|}{|F(n-2)|})\leq2\log n+2\log^2n+\gamma n\]
\end{proof}

\begin{theorem}
	$|F(n)|=|T(n)|(1+o(1))$.
\label{I3FreeDigraphs}
\end{theorem}
\begin{proof}
	Set $\eta=2^{\frac{1}{3000}}$. We will prove that there exists a constant $c\geq1$ such that for all $n$, 
\begin{equation}
|F(n)|\leq(1+c\eta^{-n})|T(n)|\label{eqn1}
\end{equation} holds. Let $n_0$ be a natural number large enough for all our estimates from Lemmas \ref{An} to \ref{Cn} to hold, and choose a $c\geq 1$ such that $|F(n)|\leq(1+c\eta^{-n})|T(n)|$ for all $n\leq n_0$. We use this as a basis for induction on $n$. 

	Suppose that for all $n'<n$ equation \ref{eqn1} holds. From Lemma \ref{Fncontained}, we have \[|F(n)|\leq|T(n)|+|A(n)|+|B(n)|+|C(n)|\] If we show that the ratio $\frac{|X(n)|}{|T(n)|}$, where $X$ is any of $A,B,C$, is at most $\frac{c}{3}\eta^{-n}$, the result will follow. We will use Lemmas \ref{An} to \ref{Cn} and induction to prove these bounds.

\begin{enumerate}
	\item{
\[
		\begin{split}
			\frac{|A(n)|}{|T(n)|}&=\frac{|A(n)|}{|F(n-1)|}\frac{|F(n-1)|}{|T(n-1)|}\frac{|T(n-1)|}{|T(n)|}\leq\\
			&\leq2^{n+\log^2n+\log n-1}(1+c\eta^{-(n-1)})6^{-\frac{1}{2}(n-1)}\leq\\
			&\leq2c2^{n+\log^2n+\log n-1}2^{-\frac{\log6}{2}(n-1)}=\\
			&=c2^{n+\log^2n+\log n-\frac{1+\alpha}{2}(n-1)}=\\
			&=c2^{n(\frac{1-\alpha}{2})+\log^2n+\log n+\frac{\alpha+1}{2}}
		\end{split}
\]
The leading term in the exponent of 2 is $n(\frac{1-\alpha}{2})$. Notice that $1-\alpha<0$, so as $n_0$ is assumed to be a very large number,
\[c2^{n(\frac{1-\alpha}{2})+\log^2n+\log n+\frac{\alpha+1}{2}}\leq\frac{c}{3}\eta^{-n}\]}
	\item{
\[
		\begin{split}
			\frac{|B(n)|}{|T(n)|}&=\frac{|B(n)|}{|F(n-\log n)|}\frac{|F(n-\log n)|}{|T(n-\log n)|}\prod_{i=1}^{\log n}\frac{|T(n-i)|}{|T(n-i+1)|}\leq\\
	&\leq2^{\beta n\log n+n+\frac{3}{2}\log^2n-\frac{1}{2}\log n}(1+c\eta^{-(n-\log n)})\prod_{i=1}^{\log n}6^{-\frac{1}{2}(n-i)}\leq\\
	&=2^{\beta n\log n+n+\frac{3}{2}\log^2n-\frac{1}{2}\log n}(1+c\eta^{-(n-\log n)})6^{-\frac{1}{2}({\sum_{i=1}^{\log n}(n-i)})}\leq\\
	&\leq2^{\beta n\log n+n+\frac{3}{2}\log^2n-\frac{1}{2}\log n}(1+c\eta^{-(n-\log n)})6^{-\frac{1}{2}(\log n(n-\frac{1}{2}\log n+1))}\leq\\
	&\leq c2^{\beta n\log n+n+\frac{3}{2}\log^2n-\frac{1}{2}\log n+1}6^{-\frac{1}{2}(\log n(n-\frac{1}{2}\log n+1))}=\\
	&=c2^{\beta n\log n+n+\frac{3}{2}\log^2n-\frac{1}{2}\log n+1+(1+\alpha)(-\frac{1}{2}(\log n(n-\frac{1}{2}\log n+1)))}
		\end{split}
\]
For readability, we will continue our calculations on the exponent of 2 until we reach a more manageable expression:
\[
\begin{split}
\beta n\log n+n&+\frac{3}{2}\log^2n-\frac{1}{2}\log n+1+(1+\alpha)(-\frac{1}{2}(\log n(n-\frac{1}{2}\log n+1)))=\\
	&=\frac{3}{2}\log^2n-\frac{1}{2}\log n+n+\beta n\log n+1-\frac{1}{2}n\log n-\frac{1}{4}\log^2n-\\
&\hspace{1cm}-\frac{1}{2}\log n-\frac{\alpha}{2}n\log n+\frac{\alpha}{2}\log^2n-\frac{\alpha}{2}\log n=\\
&=(\frac{3+\alpha}{2}-\frac{1}{4})\log^2n+n\log n(\beta-\frac{1}{2}-\frac{\alpha}{2})-\\
&\hspace{1cm}-\log n(1+\frac{\alpha}{2})+n+1=\\
&=n\log n(\beta-\frac{1}{2}-\frac{\alpha}{2})+n+\frac{5+2\alpha}{4}\log^2n-\log n(\frac{2+\alpha}{2})+1=\\
&=\frac{1-\alpha}{10^6}n\log n+n+\frac{5+2\alpha}{4}\log^2n-\log n(\frac{2+\alpha}{2})+1
\end{split}
\]
Therefore,
\[
\frac{|B(n)|}{|T(n)|}\leq c2^{\frac{1-\alpha}{10^6}n\log n+n+\frac{5+2\alpha}{4}\log^2n-\log n(\frac{2+\alpha}{2})+1}
\]
The leading term in the exponent is $\frac{1-\alpha}{10^6}n\log n$, and $1-\alpha<0$. For sufficiently large $n$, $$c2^{\frac{1-\alpha}{10^6}n\log n+n+\frac{5+2\alpha}{4}\log^2n-\log n(\frac{2+\alpha}{2})+1}<\frac{c}{3}\eta^{-n}$$
}
	\item{
\[
		\begin{split}
			\frac{|C(n)|}{|T(n)|}&=\frac{|C(n)|}{|F(n-2)|}\frac{|F(n-2)|}{|T(n-2)|}\frac{|T(n-2)|}{|T(n-1)|}\frac{|T(n-1)|}{|T(n)|}\leq\\
			&\leq2^{\gamma n+2\log n+2\log^2n}(1+c\eta^{-(n-2)})6^{-\frac{1}{2}(n-2)}6^{-\frac{1}{2}(n-1)}\leq\\
			&\leq2^{\gamma n+2\log n+2\log^2n}2c6^{-\frac{1}{2}(n-2)}6^{-\frac{1}{2}(n-1)}=\\
			&=2^{\gamma n+2\log n+2\log^2n}2c6^{-\frac{1}{2}(2n-3)}=\\
			&=c2^{\gamma n+2\log n+2\log^2n+1-\frac{\log6}{2}(2n-3)}=\\
			&=c2^{(\gamma-\log6)n+2\log n+2\log^2n+1+\frac{3}{2}\log6}
		\end{split}
\]
Now, $\gamma-\log6=1+\frac{4}{10^6}+\frac{3}{100}+\alpha(1-\frac{2}{10^6}-\frac{2}{100})-(1+\alpha)=\frac{4}{10^6}+\frac{3}{100}-\frac{2\alpha}{10^6}-\frac{2\alpha}{100}<0$, so $\frac{|C(n)|}{|T(n)|}<\frac{c}{3}\eta^{-n}$. Therefore, \[\frac{|F(n)|}{|T(n)|}\leq1+c\eta^{-n}\]}
\end{enumerate}
and we conclude that the proportion of $I_3$-free digraphs on $n$ vertices which are not bitournaments becomes negligible as $n$ tends to infinity.
\end{proof}

\bibliographystyle{acm}
\bibliography{refs}
\end{document}